\documentclass{amsart}
\usepackage[english]{babel}
\usepackage{amsthm}
\usepackage{inputenc}
\usepackage{amssymb}
\usepackage{graphicx}
\newtheorem{thm}{Theorem}[section]

\newtheorem{conj}[thm]{Conjecture}

\newtheorem{prop}[thm]{Proposition}
\newtheorem{defn}[thm]{Definition}
\newtheorem{exam}[thm]{Example}

\numberwithin{equation}{section}

\begin{document}

\title[Subalgebras of Evolution algebras]{On the property of subalgebras of Evolution algebras}%

\author{L.M. Camacho, A.Kh. Khudoyberdiyev, B.A. Omirov}

\address{[L.\ M. \ Camacho] Dpto. Matem\'{a}tica Aplicada I.
Universidad de Sevilla. Avda. Reina Mercedes, s/n. 41012 Sevilla.
(Spain)} \email{lcamacho@us.es}

\address{[A.\ Kh.\ Khudoyberdiyev and B.\ A.\ Omirov] Institute of Mathematics, National University of Uzbekistan,
Tashkent, 100125, Uzbekistan.} \email{khabror@mail.ru,
omirovb@mail.ru}

\begin{abstract}
In this paper we study subalgebras of complex finite dimensional
evolution algebras. We obtain the classification of nilpotent
evolution algebras whose any subalgebra is an evolution subalgebra
with a basis which can be extended to a natural basis of algebra.
Moreover, we formulate three conjectures related to description of
such non-nilpotent algebras.
\end{abstract}
\maketitle \textbf{Mathematics Subject Classification 2010}:
17D92, 17D99.


\section{Introduction}

Nowadays, the algebraic approach is effectively used in the study
of the genetics, dynamical systems in population biology. In 20s
and 30s of the last century the new object was introduced to
mathematics, which was the product of interactions between
Mendelian genetics and mathematics. One of the first scientist who
gave an algebraic interpretation of the $``\times"$ sign, which
indicated sexual reproduction was Serebrowsky \cite{ser}.
Etherington introduced the formal language of abstract algebra to
the study of the genetics \cite{e1}-\cite{e2}. An algebraic
approach in genetics consists of the study of various types of
genetic algebras (like algebras of free, "self-reproductive" and
bisexual populations, Bernstein algebras). Until 1980s, the most
comprehensive reference in this area was W\"orz-Busekros's book
\cite{Busekros}. A good survey on algebraic structure of genetic
inheritance is the Reed's article \cite{Reed}. More recent
results, such as genetic evolution in genetic algebras, can be
found in the Lyubich's book \cite{ly}.

Recently in the book of J.P. Tian \cite{Tian} a new type of evolution algebra was introduced. This algebra describes
some evolution laws of the genetics. The study of evolution algebras constitutes a new subject both in algebra and
the theory of dynamical systems. In the Tian's book a foundation of the framework of the theory of evolution algebras is established and some applications of evolution algebras in the theory of stochastic processes and genetics are discussed.
Evolution algebras are in general non-associative and do not belong to any of the well-known classes of non-associative algebras. In fact, nilpotency, right nilpotency and solvability might be interpreted in a biological way as a various types of vanishing (``deaths") populations. Although an evolution algebra is an abstract system, it gives an insight for the study of non-Mendelian genetics. For instance, an evolution algebra can be applied to the inheritance of organelle genes, one can predict, in particular, all possible mechanisms to establish the homoplasmy of cell populations.

Recently, Rozikov and Tian \cite{Rozikov} studied algebraic structures of evolution algebras associated with Gibbs measures defined on some graphs. In the papers \cite{Camacho2}, \cite{Casas3}, \cite{Ladra} derivations, some properties of chain of evolution algebras and dibaricity of evolution algebras were studied. Certain algebraic properties of evolution algebras (like right nilpotency, nilpotency and solvability etc.) in terms of matrix of structural constants have been investigated in \cite{Camacho1}, \cite{Casas1}, \cite{Casas2}.

It is remarkable that a subalgebra and an ideal of a genetic algebra of population, biologically can be interpreted correspondingly as a subpopulation and a dominant subpopulation with respect to mating.

This paper is devoted to study of subalgebras of finite dimensional evolution algebras.

In order to achieve our goal we organize the paper as follows. In
Section 2, we give some necessary notions and preliminary results
about evolution algebras. We consider several types of subalgebras
of evolution algebras and present examples of difference of such
subalgebras as well. Section 3 is devoted to description of
evolution algebras of permutations satisfying that any subalgebra
is evolution subalgebra with a natural basis which can be extended
to a natural basis of the algebra (condition $P$). For the list of
two-dimensional evolution algebras we identify their subalgebras.
In Section 4, we classify the nilpotent complex evolution algebras
satisfying the condition $P$. In Section 5, we formulate three
conjectures related to the description of such non-nilpotent
algebras.

Through the paper all algebras are assumed complex and finite dimensional.

\section{Preliminaries.}

In this section we give necessary definitions and preliminaries results for understanding main results of the paper.
Let us define the main object of this work - evolution algebra.

\begin{defn} \cite{Tian} \rm Let $E$ be an algebra over a field $F.$ If it admits a
basis $\{e_1, e_2, \dots\}$ such that $$e_i \cdot e_j = 0 \quad
for \ i \neq j,\qquad  e_i \cdot e_i = \sum\limits_ka_{i,k}e_k
\quad for  \ any \ i,$$ then algebra $E$ is called evolution
algebra.
\end{defn}

The basis $\{e_1, e_2, \dots\}$ is said to be {\it natural basis of evolution algebra $E$}.
It is remarkable that this type of algebra depends on natural basis $\{e_1, e_2, \dots\}.$

We denote by $A=(a_{ij})$ the matrix of the structural constants of the evolution
algebra $E$.

\begin{defn}\label{subalgebra} \cite{Tian} \rm Let $E$ be an evolution algebra and $E_1$ be a
subspace of $E.$ If $E_1$ has a natural basis $\{e_i \ | \ i \in
\Lambda_1\}$ which can be extended to a natural basis $\{e_j \ |
\ j \in \Lambda \}$ of $E,$ then $E_1$ is called an evolution
subalgebra, where $\Lambda_1$ and $\Lambda$ are index sets and
$\Lambda_1$ is a subset of $\Lambda.$
\end{defn}

In fact, for the linear subspace $E_1$ of evolution algebra $E$ we can consider three conceptions of subalgebras.
\begin{enumerate}
\item $E_1$ is subalgebra in ordinary sense;

\item $E_1$ is subalgebra and there exists a natural basis of $E_1;$

\item $E_1$ is subalgebra and there exists a natural basis of $E_1$ which can be extended to a natural basis of $E.$
\end{enumerate}

Note that Definition \ref{subalgebra} agrees with the third
conception of subalgebra.

Below we present examples which show that conceptions 1 - 3 are different in general.

\begin{exam} \rm Let $E$ be a three dimensional evolution algebra with a natural
basis $\{e_1, e_2, e_3\}$ and the table of multiplication
$$e_1\cdot e_1 = e_1+e_2, \quad e_2\cdot e_2 = -e_1- e_2, \quad e_3\cdot e_3 = e_2+e_3.$$

It is not difficult to see that $E_1 = <e_1+e_2, e_2+e_3>$ is a
subalgebra, but $E_1$ is not an evolution subalgebra (that is, there does not exist a natural basis of $E_1$).

Indeed, if we assume the contrary, i.e., in the subspace $E_1$ there exists a natural basis $\{f_1, f_2\}$, then
$$f_1 = \alpha_1 (e_1+e_2) + \alpha_2 (e_2+e_3), \quad f_2 = \beta_1 (e_1+e_2) + \beta_2
(e_2+e_3).$$
with $\alpha_1\beta_2-\alpha_2\beta_1\neq 0.$

From the condition $f_1\cdot f_2 =0$ we derive
$$\alpha_1=\alpha_2=0 \quad \mbox{or} \quad \alpha_2=\beta_2=0\quad \mbox{or} \quad \beta_1=\beta_2=0.$$

Consequently, we get a contradiction with the assumption that $\{f_1, f_2\}$ is a natural
basis of $E_1.$
\end{exam}

\begin{exam} \rm Let $E$ be a three dimensional evolution algebra with a natural
basis $\{e_1, e_2, e_3\}$ and the following table of multiplication
$$e_1\cdot e_1 = e_1+e_2+e_3, \quad e_2\cdot e_2 = -e_1- e_2 +e_3, \quad e_3\cdot e_3 = 0.$$

It is not difficult to see that $E_1 = <e_1+e_2, e_3>$ is an evolution algebra with a natural basis $\{e_1+e_2, e_3\}$, but this basis can not be extended to a natural basis of evolution algebra $E.$

If we assume that there exists a natural basis $\{f_1, f_2\}$ of
$E_1$ such that $\{f_1, f_2, f_3\}$ is a natural basis of $E$,
then
$$f_1 = \alpha_1 (e_1+e_2)+ \alpha_2 e_3, \quad f_2 = \beta_1 (e_1+e_2) + \beta_2 e_3, \quad
f_3 = \gamma_1 e_1+ \gamma_2e_2 + \gamma_3 e_3.$$

From conditions $f_1\cdot f_3 = f_2\cdot f_3 =0$ we deduce $\alpha_1=\beta_1=0$ or $\gamma_1=\gamma_2=0.$ Therefore, we get a contradiction with the assumption that $\{f_1, f_2, f_3\}$ is a basis.
\end{exam}

For the sake of convenience, we introduce the following definition.

\begin{defn} \rm An evolution algebra $E$ is said
to satisfy a condition $P$ if any subalgebra of $E$ is an
evolution subalgebra with a natural basis which can be extended to
a natural basis of $E$.
\end{defn}

In \cite{Tian} the conditions for basis transformations that preserve naturalness of the basis are given. The relation between the matrices of structure constants in a new and old natural basis is established in terms of new defined operation on matrices, as well. Since the relation is not practical for our further purposes, we give the following brief version of isomorphism.

Let us consider non-singular linear transformation $T$ of a given
natural basis $\{ e_1,\dots, e_n\}$ with a matrix $(t_{ij})_{1\leq
i,j\leq n}$ in this basis and
$$f_i= \sum_{j=1}^n t_{ij}e_j, \ 1\leq i \leq n.$$
This transformation is isomorphism if and only if $f_i\cdot f_j=0$
for all $i\neq j.$

In the following theorem we present a list (up to isomorphism) of
2-dimensional evolution algebras.
\begin{thm} \cite{Casas2} \label{thm22} Any  $2$-dimensional non-abelian evolution algebra $E$ is
isomorphic to one of the following, pairwise non-isomorphic, algebras:
\begin{enumerate}
\item $\dim E^2=1$
\begin{itemize}
\item $E_1:\ \ e_{1}e_{1} = e_{1}$,

\item $E_2: \ \ e_{1}e_{1} = e_{1}, \ \ e_{2}e_{2}= e_{1}$,

\item $E_3: \ \ e_{1}e_{1} = e_{1} + e_{2}, \ \  e_{2}e_{2}= -e_{1}-
e_{2}$,

\item $E_4: \ \ e_{1}e_{1} = e_{2}$.
\end{itemize}
\item $\dim E^{2}=2$
\begin{itemize}

\item $E_5: \ \  e_{1}e_{1}=e_{1}+a_{2}e_{2}, \ \
e_{2}e_{2}=a_{3}e_{1}+e_{2}, \ \ 1 - a_{2}a_{3}\ne 0$, where
$E_5(a_{2},a_{3})\cong E_5'(a_{3},a_{2})$,

\item $E_6: \ \ e_{1}e_{1}=e_{2}, \ \ e_{2}e_{2}=e_{1}+a_{4}e_{2}$,\\
where for $a_4\ne 0$,  $E_6(a_{4})\cong E_6(a'_{4}) \
\Leftrightarrow \ \frac{a'_4}{a_4}=\cos\frac{2\pi k}{3} + i
\sin\frac{2\pi k}{3} \ \mbox{for some} \ k=0, 1, 2$.
\end{itemize}
\end{enumerate}
\end{thm}

Consider the following $k$-dimensional evolution algebras
$$ES_{k}: \quad  \left\{\begin{array}{ll}e_i\cdot e_i = e_{i+1}, & 1 \leq i \leq k-1,\\[1mm]
e_k\cdot e_k = e_{1},&
\end{array}\right.\qquad EN_{k}: \quad \left\{\begin{array}{ll}e_i\cdot e_i = e_{i+1}, & 1 \leq i \leq k-1,\\[1mm]
e_k\cdot e_k = 0.
\end{array}\right.$$

In \cite{Khud}, the authors describe a complex evolution algebra $E_{n,\pi}(a_1, a_2, \dots, a_n)$ with a basis $\{e_1, e_2, \dots ,e_n\}$ and the table of multiplications as follows:
$$\left\{\begin{array}{ll}e_i\cdot e_i = a_ie_{\pi(i)}, &\ 1\leq i \leq n,\\[1mm]
e_i\cdot e_j = 0, & \ i \neq j,
\end{array}\right.$$
where $\pi$ is an element of the group of permutations $S_n$. Namely, the following assertion is true.
\begin{thm}  \label{thm23}
An arbitrary evolution algebra $E_{n,\pi}(a_1, a_2, \dots, a_n)$ is isomorphic to a direct sum of evolution algebras
$ES_{p_1}, \  ES_{p_2}, \  \dots, \ ES_{p_s}, \ EN_{k_1}, \
EN_{k_2}, \ \dots, \ EN_{k_r},$ i.e.,
$$E_{n,\pi}(a_1, a_2, \dots, a_n) \cong ES_{p_1}\oplus ES_{p_2} \oplus \dots \oplus  ES_{p_s} \oplus EN_{k_1} \oplus EN_{k_2} \oplus \dots \oplus EN_{k_r},$$
where $\displaystyle\sum_{i=1}^s{p_i} +\sum_{j=1}^r{k_j}=n$
\end{thm}

We introduce the following sequence:
$$E^{k}=\sum_{i=1}^{k-1}E^iE^{k-i},\, k \geq 1$$

\begin{defn}An evolution algebra $E$ is called nilpotent if there exists $n\in
\mathbb{N}$ such that $E^{n}=0$ and the minimal such number is
called index of nilpotency. \end{defn}

\begin{thm} \cite{Camacho1}\label{2.9} Let $E$ be an $n-$dimensional evolution algebra.
Then $E$ is nilpotent if and only if the matrix of structure constants $A$ can be transformed by permutation of the natural basis to the following form:
$$
A=\left(\begin{array}{ccccc}
0 &a_{12}&a_{13}&\dots & a_{1n}\\
0 &0&a_{23}&\dots & a_{2n}\\
0 &0&0&\dots & a_{3n}\\
\vdots &\vdots&\vdots&\ddots & \vdots\\
0 &0&0&\dots & 0\\
\end{array}\right).
$$
\end{thm}

The next theorem gives the classification of evolution algebras
with maximal possible index of nilpotency.

\begin{thm} \cite{Camacho1} \label{th3}
Any $n$-dimensional complex evolution algebra with maximal
index of nilpotency is isomorphic to one of pairwise non-isomorphic
algebras with the following matrix of structural constants
\[\begin{pmatrix}
 0 & 1 & a_{13} &\dots &a_{1,n-1}&0 \\[1mm]
 0 & 0 & 1 &\dots &a_{2,n-1}& 0 \\[1mm]
0 & 0 & 0 &\dots &a_{3,n-1}& 0 \\[1mm]
\vdots & \vdots & \vdots &\cdots & \vdots &\vdots\\[1mm]
0 & 0 & 0 &\cdots & 0& 1\\[1mm]
0 & 0 & 0 &\cdots &0 & 0
\end{pmatrix},
 \]
where one of non-zero $a_{ij}$ can be chosen equal to 1.
\end{thm}
The set of all evolution algebras whose matrices of structural constants have the form of Theorem
 \ref{th3} will be denoted by $ZN^{n}.$

\section{Main result}

First we investigate which evolution algebras of the list of
Theorem \ref{thm22} satisfy (or not) the condition $P$.

\begin{prop} Evolution algebras $E_1$ and $E_4$ satisfy the condition $P.$
\end{prop}

\begin{proof} Since $E$ is a two dimensional, then any non-trivial subalgebra of $E$ is one-dimensional.

Let $E_1'$ be an one-dimensional subalgebra of $E_1$ and $E_1'=<x>$ with $x=A_1e_1 + A_2e_2.$

Consider
$$x\cdot x=(A_1e_1 + A_2e_2)\cdot (A_1e_1 + A_2e_2)=A_1^2e_1$$

On the other hand,
$$ x \cdot x = \alpha x=\alpha(A_1e_1 + A_2e_2).$$

Then $A_1^2 = \alpha A_1$ and $\alpha A_2 =0.$
\begin{itemize}
\item If $\alpha =0,$ then $A_1=0$ and $\{e_2\}$ is the basis of $E_1'.$ Obviously, this basis is extendable to the natural basis $\{e_1, e_2\}$ of $E_1.$
\item If $\alpha \neq 0,$ then $A_1=\alpha, A_2=0$ and $\{e_1\}$ is the basis
of $E_1',$ which is also extendable to the natural basis of $E_1.$
\end{itemize}

The assertion of proposition regarding the algebra $E_4$ is
carried out in a similar way.
\end{proof}

\begin{prop} Evolution algebras $E_2,$ $E_3,$ $E_5$ and $E_6$ do not satisfy the condition $P.$
\end{prop}

\begin{proof}

\

\noindent\textbf{1.} Let $E_2'$ be a one-dimensional subalgebra of
$E_2$ and $E_2'=<x>$ with $x=A_1e_1 + A_2e_2.$ The equality $x
\cdot x = \alpha x$ implies $A_1^2 +A_2^2 = \alpha A_1$ and
$\alpha A_2 =0.$ We are seeking a subalgebra with a natural basis
that can not be extended to a basis of the algebra. Thus,
$\alpha=0$.

We set $A_2=iA_1$. Then $x=A_1(e_1+ie_2).$ Let us assume that the
basis $\{x\}$ can be extended to the natural basis of $E_2,$ that
is, there exists $y \in E_2$ such that $\{x, y\}$ is a natural
basis of $E_2.$ Let $y=B_1e_1 + B_2 e_2,$ then
$$0= x \cdot y = A_1(e_1+ie_2)\cdot(B_1e_1 + B_2 e_2) = A_1(B_1+iB_2)e_1.$$

Hence $B_2 = iB_1$ and we obtain a contradiction with the linear independence of elements $x$ and $y$. Thus, the evolution algebra $E_2$ does not satisfy the condition $P.$

\noindent\textbf{2.} Let $E_3'=<x>$ be a one-dimensional
subalgebra of $E_3$ with $x=A_1e_1 + A_2e_2$. Putting $A_1=A_2=1$
we conclude that $E_3'=<x>$ is a subalgebra. Let us assume that
$x=e_1+e_2$ can be extended to a natural basis of $E_3,$ then
there exists $y=B_1e_1 + B_2 e_2,$ such that $\{x, y\}$ is a
natural basis of $E_3.$

From the following equality
$$0= x \cdot y = (e_1+e_2)\cdot(B_1e_1 + B_2 e_2) =(B_1 - B_2)e_1 + (B_1 - B_2)e_2,$$
we derive $B_2 = B_1,$ which is a contradiction with condition of
$\{x, y\}$ being a basis.

Therefore, evolution algebra $E_3$ does not satisfy the condition $P.$

\noindent\textbf{3.} The element $x=A_1e_1 + A_2e_2$ forms a basis
of a one-dimensional subalgebra of $E_5$. Therefore, $\alpha x = x
\cdot x$ for some $\alpha\in \mathbb{C}$. Note, that the condition
$dim E_5=2$ implies $x \cdot x\neq0$ (consequently $\alpha\neq0$).
Without loss of generality we can assume that $\alpha=1$. Then $x
= x \cdot x$ deduce
$$\begin{cases}A_1^2+A_2^2 a_3  = A_1, \\
A_1^2a_2+A_2^2  = {A_2}. \end{cases} \eqno (3.1)
$$

It is not difficult to check that the system of equation (3.1) has a solution
$A_1, A_2$ such that $A_1A_2 \neq 0.$ Indeed, if
$a_2=a_3=0,$ then $A_1=A_2=1$ is a solution of the equation
(3.1).

Let us assume that $(a_2,a_3)\neq (0,0)$ then, without loss of generality, we can
suppose $a_3 \neq 0.$ Then from the equation (3.1) we have
$$\begin{array}{lr}
A_2 = \frac {A_1} {a_3} ((a_2a_3-1) A_1 + 1 ),&\ \ \ \ \hfill (3.2)\\
A_1^3+\frac 2 {a_2a_3-1} A_1^2 + \frac {a_3+1} {a_2a_3-1}A_1 -   \frac 1 {(a_2a_3-1)^2}=0.&\ \ \ \ \hfill (3.3)
\end{array}$$

Note that the equation (3.3) with respect to $A_1$ has three solutions and one of them does not equal to $-\frac 1 {a_2a_3-1}.$ Recall that all solutions equal to $-\frac 1 {a_2a_3-1}$ has the following cubic equation
$$A_1^3+\frac 3 {a_2a_3-1} A_1^2 + \frac {3} {(a_2a_3-1)^2}A_1 +   \frac 1 {(a_2a_3-1)^3}=0.$$

Therefore, the equation (3.1) has a solution $A_1, A_2$ with
$A_1A_2 \neq 0.$ Consequently, there exists a subalgebra $E_5' =
<x>$ with $x=A_1e_1+A_2e_2,$ where $A_1A_2 \neq 0.$

The basis of this subalgebra can not be extended to a natural
basis of $E_5.$ Indeed, if $y=B_1e_1 + B_2 e_2$ with condition
that $\{x, y\}$ is a natural basis of $E,$ then
$$0= x \cdot y = (A_1e_1+A_2e_2)\cdot(B_1e_1 + B_2 e_2) =(A_1B_1 + A_2B_2a_3)e_1 + (A_1B_1a_2 + A_2B_2)e_2,$$
which implies $$A_1B_1 + A_2B_2a_3=0,$$ $$A_1B_1a_2 + A_2B_2=0.$$

Since $A_1A_2 (1-a_2a_3) \neq 0,$ we get $B_1=B_2=0.$ It is a
contradiction with condition of $\{x, y\}$ being a basis.

Therefore, the two dimensional evolution algebra $E_5$ does not
satisfy the condition $P.$

\noindent \textbf{4.} The assertion that the algebra $E_6$ does
not satisfy the condition $P$ is carried out by applying similar
arguments as for the algebra $E_5$.
\end{proof}

%
%
%
%

Next, we present a result on preservation of the property $P$ for a direct sum of evolution algebra which satisfy the condition $P$ and abelian algebra.

\begin{prop} \label{pr3.3} Let $E$ be an $n-$dimensional evolution algebra which satisfies the condition $P$.
Then the evolution algebra $E\oplus \mathbb{C}^k$ also satisfies
the condition $P$.
\end{prop}
\begin{proof} Let $\{e_1, e_2, \dots, e_n, h_1, h_2, \dots, h_k\}$ be a
basis of $E\oplus \mathbb{C}^k$ and $M$ be an $s$-dimensional subalgebra
of $E\oplus \mathbb{C}^k.$ We set $\{x_1, x_2, \dots, x_s\}$ as a basis of
$M$ and $x_i = \sum\limits_{j=1}^n \beta_{i,j}e_j +
\sum\limits_{j=1}^k \gamma_{i,j}h_j.$

Consider $$x_i\cdot x_j =\sum\limits_{t=1}^n \beta_{i,t}\beta_{j,t}
\sum\limits_{k=1}^na_{t,k}e_k, \quad 1\leq i,j \leq s.$$

Since $x_i\cdot x_j$ belong to $M$, then the elements
$\sum\limits_{t=1}^n \beta_{i,t}\beta_{j,t}
\sum\limits_{k=1}^na_{t,k}e_k$ are expressed by linear
combinations of elements $y_i = \sum\limits_{j=1}^n\beta_{i,j}e_j,
\ 1\leq i \leq s.$ Consider $N=<y_1, y_2, \dots, y_s>$. It is easy
to see that $N$ is a subalgebra of $E$ of dimension $s'\leq s.$

For the sake of convenience, by renumeration of indexes, we can
assume that basis of $N$ is $\{y_1, y_2, \dots, y_s'\}$.

If $s'=s$, then using conditions of proposition we can find a
natural basis $\{y_1, y_2, \dots, y_{s}, z_1, z_2,\dots,$ $
z_{n-s}\}$ of $E.$ Thus the following basis $\{x_1, x_2, \dots,
x_{s}, z_1, z_2,\dots, z_{n-s}, h_1, h_2 \dots,h_{k}\}$ is a
natural basis of $E \oplus C^k.$

If $s'< s,$ then by elementary transformation of matrices we conclude
$$\left(\begin{matrix} \beta_{1,1} &\dots &\beta_{1,n} & \gamma_{1,1} &\dots
& \gamma_{1,k}\\
\beta_{2,1} &\dots &\beta_{2,n} & \gamma_{2,1} &\dots &
\gamma_{2,k}\\ \vdots &\vdots &\vdots & \vdots &\vdots & \vdots\\
\beta_{s,1} &\dots &\beta_{s,n} & \gamma_{s,1} &\dots &
\gamma_{s,k}
\end{matrix}\right) \sim \left(\begin{matrix} \beta_{1,1} &\dots &\beta_{1,n} & \gamma_{1,1} &\dots
& \gamma_{1,k}\\
\beta_{2,1} &\dots &\beta_{2,n} & \gamma_{2,1} &\dots &
\gamma_{2,k}\\ \vdots &\vdots &\vdots & \vdots &\vdots & \vdots\\
\beta_{s',1} &\dots &\beta_{s',n} & \gamma_{s',1} &\dots &
\gamma_{s',k} \\
0 &\dots & 0 & \gamma'_{s'+1,1} &\dots & \gamma'_{s'+1,k} \\
\vdots &\vdots &\vdots & \vdots &\vdots & \vdots\\ 0 &\dots &0 &
\gamma'_{s,1} &\dots & \gamma'_{s,k}
\end{matrix}\right).$$

Hence, the following elements $$x_i' =
\left\{\begin{array}{ll}\sum\limits_{j=1}^n \beta_{i,j}e_j +
\sum\limits_{j=1}^k \gamma_{i,j}h_j & 1 \leq i \leq s'
\\ \sum\limits_{j=1}^k \gamma'_{i,j}h_j & s'+1 \leq i \leq s\end{array}\right.$$
form a natural basis of $M$.

Now, we show that this basis is extendable to a natural basis of
$E \oplus \mathbb{C}^k.$ Due to $N$ being a subalgebra of $E,$ we
derive the existence of a natural basis $\{y_1, y_2, \dots,
y_{s'}, z_1, z_2,\dots, z_{n-s'}\}$ of $E.$ It is not difficult to
check that the following basis $$\{x_1', x_2', \dots, x'_{s'},
z_1, z_2,\dots, z_{n-s'}, x'_{s'+1}, x'_{s'+2}, \dots, x'_{s},
h_1', h_2' \dots, h_{k+s'-s}\}$$ is a natural basis of $E \oplus
\mathbb{C}^k,$ where $\{h_1', h_2' \dots, h_{k+s'-s}\}$ are the
complementary basis elements to $\{x'_{s'+1}, x'_{s'+2}, \dots,
x'_{s}\}$ in $\mathbb{C}^k.$
\end{proof}

Let $E$ be an $n$-dimensional evolution algebra such that $E=E_1\oplus E_2$, where $E_1$ and $E_2$ are the evolution subalgebras of $E$.

\begin{prop} \label{pr34} Let $E$ be a algebra satisfying the condition $P$. Then the subalgebras $E_1$ and $E_2$ also satisfy the condition $P$.
\end{prop}
\begin{proof} Let $E_1'$ be a subalgebra of $E_1,$ then $E_1'$ is a subalgebra of $E.$ Therefore there exist a natural basis $\{e_1', e_2', \dots, e_m'\}$ of $E_1'$ which can be extended to a natural basis
$\{e_1', e_2', \dots, e_m', x_{m+1}, x_{m+2}, \dots, x_n\}$ of
$E.$ Since $E=E_1 \oplus E_2,$ then $x_{j}=y_j+z_j$ with $y_j \in
E_1,$ $z_j \in E_2,$ $m+1\leq j\leq n.$ From $e'_i \cdot x_k =0$
and $x_k \cdot x_t =0$ we deduce $e'_i \cdot y_k=0$ and $y_k \cdot
y_t=z_k \cdot z_t=0.$ Since $\{e_1', e_2', \dots, e_m', x_{m+1},
x_{m+2}, \dots, x_n\}$ is a basis of $E,$ then any element of
$E_1$ belongs to $<e_1', e_2', \dots, e_m', y_{m+1}, y_{m+2},
\dots, y_n>.$ From the elements $y_{m+1}, y_{m+2}, \dots, y_n$ we
choose some such that $\{e_1', e_2', \dots, e_m', y_{j_1},
y_{j_1}, \dots, y_{j_k}\}$ is a basis of $E_1.$ Thus, $E_1$
satisfies the condition $P$.
\end{proof}

The next example shows that the converse assertion of Proposition \ref{pr34} is not true in general.

\begin{exam} \label{exam3.5} \rm Let $E$ be a $4-$dimensional
evolution algebra defined by a direct sum of two-dimensional evolution algebras $E_1$ and $ E_2$, where
$$E_1: e_1\cdot e_1 = e_2; \qquad E_2: e_3\cdot e_3 = e_4.$$
Clearly, $E_1$ and $E_2$ are algebras satisfying the condition $P$, but $E$  not. Indeed, the subalgebra $L=<e_1+e_3, e_2+e_4>$ is not an evolution subalgebra.
\end{exam}

In the following proposition we identify evolution algebras with the condition $P$ among the algebras of the type
$E_{n,\pi}(a_1, a_2, \dots, a_n)$.

\begin{prop} Let $E$ be an $n$-dimensional evolution algebra of the type
$E_{n,\pi}(a_1, a_2, \dots, a_n)$ which satisfies the condition $P$. Then $E$ is
isomorphic to one of the following non-isomorphic algebras:
$$ES_1 \oplus \mathbb{C}^{n-1}, \qquad EN_{s} \oplus \mathbb{C}^{n-s}, \qquad ES_1 \oplus EN_{s}\oplus \mathbb{C}^{n-s-1}.$$
\end{prop}

\begin{proof} Let $E$ be an algebra of the type $E_{n,\pi}(a_1, a_2, \dots, a_n)$, then
by Theorem \ref{thm23} we have $$E \cong ES_{p_1}\oplus ES_{p_2} \oplus \dots \oplus  ES_{p_s} \oplus EN_{k_1} \oplus EN_{k_2} \oplus \dots \oplus EN_{k_r}.$$
 Proposition \ref{pr34} we obtain that the algebras $ES_{p_i}$
and  $EN_{k_i}$ satisfy the condition $P.$

If there exists $p_j \geq 2$ with $1\leq j\leq s$ then, we have
$$ES_{p_j}: \quad  \left\{\begin{array}{ll}e_i\cdot e_i = e_{i+1}, & 1 \leq i \leq p_j-1,\\[1mm]
e_{p_j}\cdot e_{p_j} = e_{1}, & \\[1mm]
\end{array}\right.$$

This algebra does not satisfy the condition $P$, because the one-dimensional subalgebra $<x>$ with $x=e_1+e_2+\dots+e_{p_j}$ is not an evolution subalgebra. Thus, $p_j=1$ for any $j\in\{1, \dots, s\}.$

If there exist $i$ and $j$ such that $p_i=p_j=1,$ then from Example
\ref{exam3.5} we conclude that $E$ does not satisfy the condition $P.$
Therefore, we can assume $p_1=1$ and $p_j=0$ for $2 \leq j
\leq s.$

Let us suppose that there exist $i$ and $j$ such that $k_i \geq 2, k_j \geq 2.$ Without loss of generality we can suppose $i=1, \ j=2$ and $k_1 \geq k_2.$ We denote $\{e_1, e_2, \dots, e_{k_1}\}$ and $\{f_1, f_2, \dots, f_{k_2}\}$ the basis of $EN_{k_1}$ and $EN_{k_2}$, respectively. Then $M= <x_1, x_2, \dots, x_{k_2}>$ with $x_i = e_{k_1-k_2+i} + f_i, \ 1 \leq i \leq k_2$ form a subalgebra of $E$ with the following products $$x_i \cdot x_i = x_{i+1}, \quad 1 \leq i \leq k_2-1, \qquad x_{k_2} \cdot x_{k_2} =0.$$

It is not difficult to check that $M$ is not an evolution subalgebra. Thus, we get a contradiction with the assumption that there exist $i$ and $j$ such that $k_i \geq 2, k_j \geq 2.$ Therefore, we can assume $k_j=1$ for $2 \leq j \leq r.$

Since $EN_1$ is a one-dimensional algebra with trivial
multiplication, then by Proposition \ref{pr3.3} it is enough to
consider the case $s=r=1,$ that is, we reduce the study to
$ES_{p}\oplus EN_{k}$ with $ p\in\{0, 1\}.$

\begin{itemize} \item In the case of $p=1$ and $k=1$ we obtain the algebra $ES_{1}\oplus
\mathbb{C}^{n-1};$
\item In the case of $p=1$ and $k\geq 2$ we obtain the algebra
$ES_{1}\oplus EN_{k} \oplus \mathbb{C}^{n-k-1};$
\item In the case of $p=0,$ we obtain the algebra $EN_{k} \oplus
\mathbb{C}^{n-k}.$
\end{itemize}

It is not difficult to check that all obtained algebras
$ES_1 \oplus \mathbb{C}^{n-1}, \ EN_{s} \oplus \mathbb{C}^{n-s}, \ ES_1 \oplus
EN_{s}\oplus \mathbb{C}^{n-s-1}$ satisfy the condition $P$.
\end{proof}

\section{Nilpotent case.}

Let $E$ be an $n$-dimensional non-abelian evolution algebra with a natural basis
$\{e_1, e_2, \dots, e_n\}.$  By transformation of the basic elements we get the following table of multiplication
$$e_i^2 \neq 0, \quad \ 1 \leq i \leq k, \quad \quad e_i^2 = 0, \quad k+1 \leq i \leq n, \quad k\leq n. \eqno(4.1)$$

We consider the notation given in Theorem \ref{2.9}.
\begin{prop} \label{pr4.1} Let $rank(A)<k.$ Then $E$ does not
satisfy the condition $P.$
\end{prop}

\begin{proof}  We shall prove the statement of proposition by the contrary. Let us assume that $rank(A) =s < k,$ then there exist the indexes $i_1, \ i_2, \dots, i_s$ such that the elements $e_{i_1}^2, e_{i_2}^2, \dots, e_{i_s}^2$ are linearly independent. For the sake of convenience we shall assume that
$e_{1}^2, e_{2}^2, \dots, e_{s}^2$ are linearly independent.

Consider the non-trivial linear combination $$\alpha_1 e_{1}^2 + \alpha_2 e_{2}^2 +  \dots +
\alpha_s e_{s}^2 + \alpha_{s+1} e_{s+1}^2=0.$$
Since $\alpha_{s+1} \neq 0$ (otherwise we obtain trivial linear combination) we get $$e_{s+1}^2 = - \frac{\alpha_1}
{\alpha_{s+1}}e_{1}^2 - \frac{\alpha_2} {\alpha_{s+1}} e_{2}^2 -
\dots - \frac{\alpha_s} {\alpha_{s+1}} e_{s}^2.$$

Due to existence $\alpha_i \neq 0$ for some $1\leq i \leq s,$ without loss of generality, we can assume
$\alpha_1 \neq 0.$

For the element $x = \sqrt{\alpha_1} e_{1} +
\sqrt{\alpha_2} e_{2} + \dots + \sqrt{\alpha_s} e_{s} +
\sqrt{\alpha_{s+1}} e_{s+1}$ we have $x\cdot x = 0.$ Hence, $<x>$ is
an one-dimensional subalgebra. Consequently, there exist a natural basis
$\{x, y_2, y_3, \dots, y_n\}$ of $E.$

Let us introduce the following denotations $$y_i = \sum\limits_{j=1}^n\beta_{i,j}e_j, \quad 2 \leq i \leq
n.$$

Consider
$$0=x\cdot y_i = (\sum\limits_{j=1}^{s+1}\sqrt{\alpha_j} e_{j})\cdot (\sum\limits_{j=1}^n\beta_{i,j}e_j) =
\sum\limits_{j=1}^{s+1}\sqrt{\alpha_j} \beta_{i,j}e_j^2=$$
$$\sum\limits_{j=1}^{s}\sqrt{\alpha_j} \beta_{i,j}e_j^2 -
\sqrt{\alpha_{s+1}}
\beta_{i,s+1}\sum\limits_{j=1}^{s}\frac{\alpha_j}
{\alpha_{s+1}}e_{j}^2 = \sum\limits_{j=1}^{s}\big(\sqrt{\alpha_j}
\beta_{i,j} - \sqrt{\alpha_{s+1}} \beta_{i,s+1}\frac{\alpha_j}
{\alpha_{s+1}} \big)e_j^2 .$$

Thus, $$\sqrt{\alpha_j} \beta_{i,j} - \sqrt{\alpha_{s+1}}
\beta_{i,s+1}\frac{\alpha_j} {\alpha_{s+1}}=0, \quad 2\leq i \leq
n, \ 1 \leq j \leq s. \eqno(4.2)$$

For $j=1$ in the restrictions (4.2) we obtain
$$\beta_{i,s+1} = \sqrt{\frac{\alpha_{s+1}}{\alpha_1}} \beta_{i,1}, \quad 2\leq i \leq n.$$

We have that $\{x, y_2, y_3, \dots, y_n\}$ and $\{e_1, e_2, \dots,
e_n\}$ are two bases of $E.$ Then the matrix of change of basis
has the following form:
$$B=\left(\begin{matrix} \sqrt{\alpha_1}&  \dots&
\sqrt{\alpha_s}& \sqrt{\alpha_{s+1}}& 0 & \dots & 0
\\[1mm] \beta_{2,1}& \dots& \beta_{2,s} &
\sqrt{\frac{\alpha_{s+1}}{\alpha_1}} \beta_{2,1}  & \beta_{2,s+2}& \dots&  \beta_{2,n}\\[1mm]
 \beta_{3,1}& \dots& \beta_{3,s} &
\sqrt{\frac{\alpha_{s+1}}{\alpha_1}} \beta_{3,1}  & \beta_{3,s+2}&
\dots& \beta_{3,n} \\[1mm]
\vdots&\vdots&\vdots&\vdots&\vdots&\vdots&\vdots\\[1mm] \beta_{n,1}& \dots& \beta_{n,s} &
\sqrt{\frac{\alpha_{s+1}}{\alpha_1}} \beta_{n,1}  & \beta_{n,s+2}&
\dots& \beta_{n,n}
\end{matrix}
\right).$$

Since $det(B) =0$ we get a contradiction. Thus, the algebra $E$ does not satisfy the condition $P$.
\end{proof}

In the following theorem we describe nilpotent evolution algebras which satisfy the condition $P$.

\begin{thm} An arbitrary nilpotent evolution algebra satisfying the condition $P$ is isomorphic to $$\widetilde{E} \oplus \mathbb{C}^{n-k},$$ where $\widetilde{E}\in ZN^{k}.$
\end{thm}
\begin{proof}
Let $E$ be a nilpotent evolution algebra satisfying the condition $P$ with the table of multiplication (4.1).
Then the matrix $A$ has the form:
$$A= \left(\begin{matrix} 0& a_{1,2} &a_{1,3}& \dots & a_{1,k+1} &  \dots & a_{1,n}
\\[1mm] 0& 0& a_{2,3} & \dots & a_{2,k+1} & \dots  & a_{2,n}\\[1mm]
\vdots&\vdots&\vdots&\vdots&\vdots&\vdots &\vdots \\[1mm] 0& 0& 0 & \dots  & a_{k,k+1} & \dots  & a_{k,n}\\[1mm]
0& 0& 0 & \dots & 0 & 0 & 0 \\[1mm]
\vdots&\vdots&\vdots&\vdots&\vdots&\vdots &\vdots \\[1mm]
0& 0& 0 & \dots & 0 & 0 & 0
\end{matrix}
\right).$$

Putting $e_{k+1}' = \sum\limits_{j=k+1}^na_{k,j}e_j$ we can assume $e_k^2=e_{k+1},$ that is, we can always suppose $a_{k, k+1}=1$ and $a_{k, j}=0$ for $k+2 \leq j \leq n.$

Let $a_{1,2}a_{2,3}\dots a_{k-1,k}=0$ be. Then we denote by $t$ the greatest number such that $a_{t,t+1}=0,$ i.e.,
 $a_{i,i+1}\neq 0$ for $t+1 \leq i \leq k+1.$ If $a_{i,i+1}=0$ for all $1 \leq i \leq k-1,$ then we put $t=k.$

Consider the subalgebra $E_1=<e_t + e_{t+1}, e_{t+2}, \dots, e_n>.$ Then there exists a natural basis $\{y_1,$ $ y_2,
\dots,$ $ y_t,$ $ e_t + e_{t+1}, e_{t+2}, \dots, e_n\}$ of $E.$

We set $y_i = \sum\limits_{j=1}^n\beta_{i,j}e_j$ with $1 \leq i \leq t.$
Then
$$0=(e_t+e_{t+1})\cdot y_i = \beta_{i,t}e_t^2+ \beta_{i,t+1}e_{t+1}^2.$$
Due to Proposition \ref{pr4.1} we conclude $rank(A) = k.$ It
implies that $e_t^2$ and $e_{t+1}^2$ are linearly independent.
Therefore, $\beta_{i,t} = \beta_{i,t+1} =0, \ 1 \leq i \leq t.$ We
have two bases in $E:$  $\{x, y_2, y_3, \dots, y_n\}$ and $\{e_1,
e_2, \dots, e_n\}.$  Then, the matrix of changes of basis has the
following form:
$$B=\left(\begin{matrix} \beta_{1,1}& \dots& \beta_{1,t-1} &
0  & 0 & \beta_{1,t+2}& \dots& \beta_{1,n}\\[1mm]
\vdots&\vdots&\vdots&\vdots&\vdots&\vdots&\vdots&\vdots\\[1mm]
\beta_{t,1}& \dots& \beta_{t,t-1} & 0  & 0 & \beta_{t,t+2}& \dots&
\beta_{t,n} \\[1mm]
0& \dots& 0 & 1  & 1 & 0& \dots& 0 \\[1mm]
0& \dots& 0 & 0  & 0 & 1& \dots& 0\\[1mm]
\vdots&\vdots&\vdots&\vdots&\vdots&\vdots&\vdots&\vdots\\[1mm]
0& \dots& 0 & 0  & 0 & 0& \dots& 1\\
\end{matrix}
\right).$$
Hence, $det(B) =0$ and we get a contradiction. Therefore, $a_{1,2}a_{2,3}\dots a_{k-1,k}\neq 0.$

Taking the following change of basis:
$$\left\{\begin{array}{ll}
e_1' = a_{1,2}^{-1/2}a_{2,3}^{-1/4} \dots a_{k-1,k}^{-1/2^{k-1}}e_1,\\[1mm]
e_2' = a_{2,3}^{-1/2}a_{3,4}^{-1/4} \dots
a_{k-1,k}^{-1/2^{k-2}}e_2,\\[1mm]
\dots\dots\dots\dots\dots\dots \\[1mm]
e_{k-1}' = a_{k-1,k}^{-1/2}e_{k-1},& k \leq i \leq n,\\[1mm]
e_{i}' = e_{i},
\end{array}\right.$$
 we can suppose $a_{1,2} = a_{2,3}= \dots = a_{k-1,k}=1.$

Moreover, the basis transformation
$$e_j''=e_j'+\sum\limits_{i=k+2}^na _{j-1, i}e_i' +
\sum\limits_{i=k+2}^n\left( \sum\limits_{t=j}^{k-1}
a_{t,i}\left(\sum\limits_{p=1}^{t-j+1}(-1)^p \prod\limits_{h=1}^p
a_{j-2+h, t+1-p+h}\right)\right)e_i', \quad 2 \leq j \leq k,
$$
implies that the algebra $E$ belongs to the family of algebras
$ZN^{k+1} \oplus \mathbb{C}^{n-k-1}.$ Taking into account the result of Proposition \ref{pr3.3}
it is enough to prove that any evolution algebra of the set $ZN^{n}$ satisfies the condition $P.$

Indeed, if a subalgebra $M$ of $\widetilde{E}$ (where
$\widetilde{E}\in ZN^{k}$) contains an element
$e_j+\sum\limits_{s=j+1}^k\beta_{s}e_s$,  then $\{e_j, e_{j+1},
\dots, e_k\} \subseteq M.$ Hence, the algebra $\widetilde{E}$ has
only subalgebras of the form $E_i =<e_i, e_{i+1}, \dots, e_k>.$ It
is not difficult to see that the subalgebras $E_i$ are the
evolution subalgebras of $\widetilde{E}.$

\end{proof}

\section{Conjectures}

In this section we formulate two related conjectures. The positive answer to the first conjecture implies a positive answer for the second one. In fact, the correctness of the second conjecture close the description of evolution algebras which satisfy the condition $P$.

\begin{conj} \label{con1} Let $A = (a_{i,j})_{1\leq i,j\leq n}$ be a complex
invertible matrix. Then the following system of equations $$\left(\begin{matrix}x_1^2\\ x_2^2 \\ \vdots\\
x_n^2 \end{matrix}\right) =\left(\begin{matrix}a_{1,1} & a_{2,1}& \dots & a_{n,1}\\
a_{1,2} & a_{2,2}& \dots & a_{n,2}\\
\vdots & \vdots& \dots & \vdots \\a_{1,n} & a_{2,n}& \dots & a_{n,n}\\
\end{matrix}\right) \left(\begin{matrix}x_1\\ x_2 \\ \vdots\\ x_n
\end{matrix}\right) \eqno (5.1)$$
has a solution $(x_{1}, x_2, \dots, x_n)$ such that $x_i \neq 0$ for
all $i.$
\end{conj}

\begin{itemize}
\item If $n=1,$ then this conjecture is evidently true.

\item If $n=2,$ then we consider subcases:

\begin{enumerate}
\item Subcase $(a_{1,2},a_{2,1})=(0,0).$ Then $a_{1,1}a_{2,2}\neq 0$ and we have a
solution $x_1=a_{1,1}, x_2 =a_{2,2}.$

\item Subcase $(a_{1,2}, a_{2,1})\neq (0,0).$ Then, without loss of generality,
we can assume $a_{1,2}\neq 0.$ Putting $x_2 = \frac 1 {a_{1,2}}(x_1^2 - a_{1,1}x_1),$ we get
$$x_{1}(x_1^3-2a_{1,1}x_1^2+(a_{1,1}^2-a_{1,2}a_{2,2})x_1+a_{1,2}(a_{1,2}a_{2,1}-a_{2,2}a_{1,1}))=0. \eqno (5.2)$$

Since $a_{1,2}(a_{1,2}a_{2,1}-a_{2,2}a_{1,1}) \neq 0,$ the equation (5.2) has three non-trivial solution.
Moreover,
$$x_1^3-2a_{1,1}x_1^2+(a_{1,1}^2-a_{1,2}a_{2,2})x_1+a_{1,2}(a_{1,2}a_{2,1}-a_{2,2}a_{1,1})\neq
(x-a_{1,1})^3.$$
From this inequality we deduce that equation (5.2) has a
solution $x_1$ different from $0$ and $a_{1,1}.$ Hence, $x_2 = \frac 1 {a_{1,2}}
(x_1^2 - a_{1,1}x_1)\neq 0.$

Thus, Conjecture \ref{con1} for the case  $n=2$ is correct, as well.
\end{enumerate}
\end{itemize}

Now we present two consequences of Conjecture \ref{con1} about the description of evolution algebras satisfying the condition $P$.

\begin{conj} \label{con2}  Let $E$ be an $n$-dimensional ($n\geq 2$) evolution algebra
with natural basis $\{e_1, e_2, \dots, e_n\}$ and an invertible
matrix $A.$ Then $E$ does not satisfy the condition $P.$
\end{conj}

Indeed, if we consider $x\cdot x = x$ with $x=\sum\limits_{i=1}^nx_ie_i,$ then comparing the coefficients at the basic elements $e_i$, we obtain the system of equations (5.1). Due to $det A \neq 0$ and according to Conjecture \ref{con1} we get the existence of a solution $(x_{1}, x_2, \dots, x_n)$ such that $x_i \neq 0$ for all $i.$ Therefore, $E_1=<x>$ is a subalgebra of $E.$ However, this subalgebra is not an evolution subalgebra and the assumption of Conjecture \ref{con2} is correct.

\begin{conj} \label{con3} Let $E$ be an $n$-dimensional non-nilpotent evolution algebra
which satisfying the condition $P$. Then  $E$ is isomorphic to
one of the following, pairwise non-isomorphic, algebras:
$$ES_1 \oplus \mathbb{C}^{n-1}, \qquad ES_1 \oplus \widetilde{E}\oplus \mathbb{C}^{n-s-1},$$
 where $\widetilde{E}\in ZN^{s}$ is a nilpotent evolution algebra with maximal index of
nilpotency.
\end{conj}

{\bf Explanation of Conjecture \ref{con3}.}

Let $E$ be an $n$-dimensional non-nilpotent evolution algebra  satisfying the condition $P$ and with the table of multiplication (4.1).

Note that the table of multiplication (4.1) for $k=1$ give the algebra $ES_1 \oplus \mathbb{C}^{n-1}.$ Therefore, further we shall assume $k \geq 2.$

Let us introduce the denotations
$x_{s,t} = (a_{s,1},a_{s,2}, \dots, a_{s,t})$ with $1 \leq s \leq t$ and $1 \leq t\leq k.$

Note that there are no $s'$ and $s''$ such that $x_{s',k} =
x_{s'',k} =  (0,0, \dots, 0).$ In fact, if there exist $s'$ and
$s''$, then the subalgebra $E_1 = <e_{s'} +e_{s''}, e_{k+1},
\dots, e_n>$ is not an evolution subalgebra.

It is not difficult to see that the non-zero vectors $x_{s_1,k}, x_{s_2,k}, \dots,
x_{s_t,k}$ are linearly independent. Otherwise there exist a non-trivial linear combination
$$\alpha_1x_{s_1,k} + \alpha_2 x_{s_2,k} +  \dots +  \alpha_t x_{s_t,k}=0,$$
and the subalgebra $E_1 = <\sqrt{\alpha_1} e_{s_1} +\sqrt{\alpha_2} e_{s_2} +\dots+ \sqrt{\alpha_t} e_{s_t},
e_{k+1}, e_{k+2}, \dots, e_{n}>$ is not an evolution subalgebra.

\textbf{Iteration 1.} Let us assume that all the vectors $x_{s,k}$ are non-zero (there are $k$-pieces), then the determinant of the main minor of the order $k$ is non-zero.

Then taking the change $e_i' = e_i +
\sum\limits_{j=k+1}^n\beta_{i, j} e_j, \ 1 \leq i \leq k,$ where $\beta_{i, j}$ can be find from the following equation
$$ \left(\begin{matrix}a_{1,1} & a_{1,2}& \dots & a_{1,k}\\
a_{2,1} & a_{2,2}& \dots & a_{2,k}\\
\vdots & \vdots& \dots & \vdots \\a_{k,1} & a_{k,2}& \dots & a_{k,k}\\
\end{matrix}\right) \left(\begin{matrix}
\beta_{1,k+1} & \beta_{1,k+2} & \dots & \beta_{1,n} \\
\beta_{2,k+1} & \beta_{2,k+2} & \dots & \beta_{2,n} \\
\vdots & \vdots & \vdots & \vdots \\
\beta_{k,k+1} & \beta_{k,k+2} & \dots & \beta_{k,n}
\end{matrix}\right) = \left(\begin{matrix}
a_{1,k+1} & a_{1,k+2}& \dots & a_{1,n} \\
a_{2,k+1} & a_{2,k+2}& \dots & a_{2,n} \\
\vdots & \vdots & \vdots & \vdots \\
a_{k,k+1} & a_{k,k+2}& \dots & a_{k,n}
\end{matrix}\right), \eqno (5.3)$$
we obtain that the evolution algebra $E$ is isomorphic to the algebra
$E' \oplus \mathbb{C}^{n-k}.$ The basis $\{e'_1, e'_2, \dots, e'_k\}$ is a natural basis of the evolution algebra $E'$. Due to Proposition \ref{pr3.3} the evolution algebra $E'$ should satisfy the condition $P$, but according to Conjecture \ref{con2} the algebra $E'$ does not satisfy the condition $P$. Thus, in this case we get a contradiction.

Let us suppose that there exists some $s_0$ such that $x_{s_0,k}=(0,0, \dots, 0).$ Without loss of generality, we can suppose $s_0=k$. Then we obtain the multiplication
$$e_i\cdot e_i = \sum\limits_{i=1}^n a_{i,j}e_i,\quad 1 \leq i \leq k-1, \quad e_k\cdot e_k = \sum\limits_{i=k+1}^n a_{i,j}e_i, \quad e_i\cdot e_i =0, \quad k+1 \leq i \leq n.$$

Applying a change of basis similar to (5.3) we can suppose $a_{i,j}
=0$ for $1 \leq i \leq k-1, \ k+1 \leq j \leq n.$ In addition, choosing $e_{k+1}' = \sum\limits_{i=k+1}^n a_{i,j}e_i,$ we derive $e_k\cdot e_k =e_{k+1}.$

\textbf{Iteration 2.} Now we consider the vectors $x_{s,k-1} = (a_{s,1},a_{s,2}, \dots,
a_{s,k-1}),$ for $1 \leq s \leq k-1.$


Now we reduce our study to the case when all vectors $x_{s,k-1}$ are non-zero. Then the main minor of order $k-1$ is non-zero and the equality $x\cdot x = x$ with $x=\sum\limits_{i=1}^{k+1}x_ie_i$ implies the following system of
equations
$$\left(\begin{matrix}a_{1,1} & a_{2,1}& \dots & a_{k-1,1} & 0&0\\
a_{1,2} & a_{2,2}& \dots & a_{k-1,2}& 0&0\\
\vdots & \vdots& \dots & \vdots  & \vdots & \vdots \\a_{1,k} & a_{2,k}& \dots & a_{k-1,k}& 0&0\\
0& 0& \dots & 0& 1&0\\
\end{matrix}\right) \left(\begin{matrix}x_1^2\\ x_2^2 \\ \vdots\\
x_{k}^2 \\
x_{k+1}^2
\end{matrix}\right) = \left(\begin{matrix}x_1\\ x_2 \\ \vdots\\
x_k\\
x_{k+1}
\end{matrix}\right).$$

From Conjecture \ref{con1} we have the existence of solution $x_i\neq0$ of the system of equation
$$\left(\begin{matrix}a_{1,1} & a_{2,1}& \dots & a_{k-1,1} \\
a_{1,2} & a_{2,2}& \dots & a_{k-1,2} \\
\vdots & \vdots& \dots & \vdots  \\a_{1,k-1} & a_{2,k-1}& \dots & a_{k-1,k-1}\\
\end{matrix}\right) \left(\begin{matrix}x_1^2\\ x_2^2 \\ \vdots\\
x_{k-1}^2 \end{matrix}\right) = \left(\begin{matrix}x_1\\ x_2 \\ \vdots\\
x_{k-1}
\end{matrix}\right)$$
and $x_{k}=\sum\limits_{s=1}^{k-1}a_{s,k}x_{s}^2, \ x_{k+1} = x_k^2.$
Therefore, the element $x$ is not extendable to a natural basis of the evolution algebra $E$. We get a contradiction with the assumption that all vectors $x_{s,k-1}$ are non-zero.

Continuing with the iterations for the vectors $x_{s,k-2}, \
x_{s,k-3}, \dots x_{s,2},$ we conclude that for all $t$ there
exists $s_t$ such that $x_{s_t,t}=(0, 0, \dots, 0).$ By shifting
basis elements we can assume that $s_t=t$ and we obtain that the
evolution algebra $E$ is isomorphic to the following algebra:
$$\begin{array}{lll}
e_1\cdot e_1 = \sum\limits_{j=1}^k a_{i,j}e_i, & e_i\cdot e_i = \sum\limits_{j=i+1}^k a_{i,j}e_i,& 2 \leq i \leq k-1,\\[2mm]
e_{k}\cdot e_{k} =  e_{k+1}, &  e_i\cdot e_i =0, & k+1 \leq i \leq n.
\end{array}$$

For the element $x=\sum\limits_{i=1}^{k+1}x_ie_i$ the equality $x\cdot x = x$ implies the system of equations as follows
$$\left\{\begin{array}{l} a_{1,1}x_1^2= x_1,\\
a_{1,2}x_1^2= x_2,\\
a_{1,3}x_1^2+a_{2,3}x_2^2= x_3,\\
\dots\dots\dots\dots\dots\dots \\
a_{1,k}x_1^2+a_{2,k}x_2^2 +\dots +a_{k-1,k}x_{k-1}^2= x_k,\\
x_{k}^2= x_{k+1}.
\end{array}\right.$$

Taking into account that the algebra $E$ is non-nilpotent, we have $a_{1,1}\neq0$ and $x_1 = \frac
{1} {a_{1,1}}.$

If $(a_{1,2}, a_{1,3}, \dots, a_{1, k}) \neq (0,0,
\dots, 0),$ then there exists a solution $(x_1, \dots x_{k+1})$
such that $x_i \neq 0$ for some $2 \leq i \leq k+1.$
Similarly as above we conclude that the evolution algebra $E$ does not satisfy the condition $P$.

Thus, we get $(a_{1,2}, a_{1,3}, \dots, a_{1, k}) = (0,0, \dots,
0).$ Hence  the $n$-dimensional non-nilpotent evolution algebra $E$
satisfying the condition $P$ is isomorphic to one of the
following, pairwise non-isomorphic, algebras:
$$ES_1 \oplus C^{n-1}, \qquad ES_1 \oplus \widetilde{E}\oplus C^{n-s-1}, \qquad \widetilde{E}\in ZN^{s}.$$

\section*{ Acknowledgements}
The authors are grateful to Professor A.S. Dzhumadil'daev for
formulation of the problem considered in the paper. This works is
supported by the Grant No.0251/GF3 of Education and Science
Ministry of Republic of Kazakhstan. The second named author was
partially supported by IMU/CDC-program, and would like to
acknowledge the hospitality of the  Instituto de Mat\'{e}maticas
de la Universidad de Sevilla (Spain).


\begin{thebibliography}{20}

\bibitem{Camacho1} Camacho L.M., G\'omez J.R., Omirov B.A., Turdibaev R.M. \emph{Some properties of evolution algebras,}
Bull. Korean Math. Soc., vol. 50(5),  2013, p. 1481--1494.

\bibitem{Camacho2} Camacho L.M., G\'omez J.R., Omirov B.A., Turdibaev R.M. \emph{The derivations of some evolution algebras,} Linear Mult. Alg., vol. 61(3), 2012, p. 309--322.

\bibitem{Casas1} Casas J.M., Ladra M., Omirov B.A., Rozikov U.A. \emph{On nilpotent index and dibaricity of evolution algebras,} Linear Algebra Appl., vol. 439, 2013, p. 90--105.

\bibitem{Casas2} Casas J.M., Ladra M., Omirov B.A., Rozikov U.A. \emph{On evolution algebras,} Algebra Colloquium, vol. 21(2), 2014, p. 331-342.

\bibitem{Casas3} Casas J. M., Ladra M., Rozikov U. A. \emph{Chain of evolution algebras,} Linear Algebra Appl., vol. 435(4), 2011, p. 852--870.

\bibitem{e1} Etherington I.M.H., Genetic algebras, Proc. Roy. Soc. Edinburgh, vol. 59, 1939, p. 242--258.

\bibitem{e2} Etherington I.M.H., Non-associative algebra and the symbolism of genetics, Proc. Roy. Soc. Edinburgh, vol. 61, 1941, p. 24--42.

\bibitem{Khud} Khudoyberdiyev A. Kh., Omirov B.A., Qaralleh I. \emph{Few remarks on evolution algebras,} arXiv:1307.0993v1.

\bibitem{Ladra} Ladra M., Omirov B.A., Rozikov U.A. \emph{On dibaric and evolution algebras,} arXiv:1104.2578v1.

\bibitem{ly} Lyubich Y.I. \emph{Mathematical structures in population genetics}, Springer-Verlag, Berlin, 1992.

\bibitem{Reed} Reed M.L. \emph{Algebraic structure of genetic inheritance,} Bull. Amer. Math. Soc. (N.S.), vol. 34(2), 1997, p. 107--130.

\bibitem{Rozikov} Rozikov U.A., Tian J.P. \emph{Evolution algebras generated by Gibbs measures}, Lobachevskii Jour. Math., vol. 32(4), 2011, p. 270--277.

\bibitem{ser} Serebrowsky A. \emph{On the properties of the Mendelian equations,} Doklady A.N.SSSR, vol. 2, 1934, p. 33--36
(in Russian).

\bibitem{Tian} Tian J.P. \emph{Evolution algebras and their applications}, Lecture Notes in Math., 1921. Springer, Berlin,
2008.

\bibitem{Busekros} W\"{o}rz-Busekros A. \emph{Algebras in genetics,} Lecture Notes in Biomathematics 36,
Springer-Verlag, Berlin-New York, 1980.

\end{thebibliography}
\end{document}